\documentclass{article}
\usepackage{amsmath, amsthm}
\usepackage{amssymb,pb-diagram,pst-all}
\usepackage{amscd}
\usepackage[pdftex]{graphicx}
\usepackage{tikz}

\newcommand{\RR}{\mathbb R}

\newcommand{\PP}{\mathbb P}

\newcommand{\CC}{\mathbb C}

\newcommand{\mcL}{\mathcal L}

\newtheorem{thm}{Theorem}[section]
\newtheorem{cor}{Corollary}[section]
\newtheorem{prop}{Proposition}[section]
\newtheorem{lem}{Lemma}[section]
\newtheorem{defin}{Definition}[section]
\newtheorem{exmple}{Example}[section]
\newtheorem{rem}{Remark}[section]
\newtheorem{qz}{Question}[section]

\newenvironment{example}{\begin{exmple}\rm }{\end{exmple}}
\newenvironment{remark}{\begin{rem}\rm }{\end{rem}}

\begin{document}

\title{Two-graphs and the embedded topology of smooth quartics and its bitangent lines}
\author{Shinzo Bannai and Momoko Yamamoto-Ohno}
\date{\today}

\maketitle

\section{Introduction}

Smooth plane quartics and their bitangent lines have attracted the interest of many mathematicians from various points of views. In this paper, we study arrangements consisting of a smooth quartic and its bitangent lines with regards to their embedded topology in the context of Zariski pairs, the concept of which was first introduced in \cite{artal}.

Let $C_1, C_2\subset\PP^2$ be plane curves. Then $C_1$ and $C_2$ are said to be a Zariski pair if it satisfies the following two conditions:
\begin{enumerate}
\item There exist tubular neighborhoods $T(C_i)$ of $C_i$ $(i=1,2)$ such that the pairs $(T(C_1),C_1)$ and $(T(C_2),C_2)$ are homeomorphic as pairs.
\item The pairs $(\PP^2, C_1)$ and $(\PP^2, C_2)$ are not homeomorphic as pairs.
\end{enumerate}

The notion of a Zariski pair can be generalized to any number of curves $C_1,\ldots, C_r$, where we require any two curves to be Zariski pairs. The first condition above is equivalent to saying the $C_1$ and $C_2$ have the same combinatorial type which can be calculated algebraically. Thus a Zariski pair is a pair of plane curves having similar algebraic properties however with subtle differences which contribute to the difference in the embedded topology. Hence, the key in studying Zariski pairs is to find a suitable method to detect the subtle differences. Many invariants have been used to this end such as the fundamental groups  of the complements $\pi(\PP^2\setminus C_i)$, the Alexander polynomials $\Delta_{C_i}(t)$ and the existence/non-existence of certain Galois covers branched along $C_i$ (See \cite{act} for a survey on these topics). Also, other more recent types of invariants called \lq\lq splitting invariants" have been developed and been proved effective in studying reducible curves (\cite{artal-tokunaga}, \cite{bannai2016}, \cite{shirane16}, \cite{shirane17}). However, as the number of irreducible components increase, these invariants become more increasingly complex, and it becomes difficult to manage the data they provide. One approach in resolving this problem is to consider all sub-arrangements of the curves as done in \cite{bannai-tokunaga}, \cite{bgst}.  In this paper, we take another step further in this direction and apply the comcept of  two-graphs and switching classes from graph theory (see Section \ref{sec:two-graphs} for definitions) in order to clarify the data of the splitting invariants, which enables us to find a new Zariski 5-ple of degree 9 and a 9-ple of degree 10.

Let $Q$ be a smooth plane quartic. It is well known that a smooth plane quartic has 28 bitangent lines $L_1,\ldots, L_{28}$. For a subset $I\subset\{1,\ldots,28\}$, let
\[
Q+L_I=Q+\sum_{i\in I}L_i .
\]
Considering the curves of the above form, the case where $|I|=3$ has been studied by E. Artal-Bartolo and J. Vall\`es where they found a Zariski pair, about which the authors were informed via private communication. The case where $|I|=4$ has been treated by the authors together with H. Tokunaga where a Zariski triple was found (\cite{ban-yam-tok2018}). In this paper we consider the cases where $|I|\geq 5$.

For each $I$, we can associate a two-graph or equivalently a switching class $\mcL_I$ to $Q+L_I$ as will be explained in Section \ref{sec:oursetting}.  Our first main result is the following.
\begin{thm}\label{thm:main1}
Let $I_1, I_2\subset \{1,\ldots, 28\}$ and let $\mcL_{I_1}$, $\mcL_{I_2}$ be their associated two-graphs. If $\mcL_{I_1}$ and $\mcL_{I_2}$ are not equivalent as two-graphs, then $(\PP^2, Q+L_{I_1})$ and $(\PP^2, Q+L_{I_2})$ are not homeomorphic as pairs. If furthermore $Q+L_{I_1}$ and $Q+L_{I_2}$ have the same combinatorial type, then they form a Zariski pair.
\end{thm}

\begin{remark}
Two-graphs are well suited for describing the geometric conditions involved where as switching classes are more suited in executing the actual computations.
\end{remark}

The number $t_n$ of equivalence classes of two-graphs with  $n$ vertices is known (\cite{mallows-sloane}). The values for small $n$ are given  in the following table: 
\begin{center}
\begin{tabular}{|c|c|c|c|c|c|c|c|c|c|c|c|}
\hline
$n$ & 1 & 2 & 3 & 4 & 5 & 6 & 7 & 8 & 9 & 10 & $\cdots$ \\
\hline
$t_n$& 1 & 1 & 2 & 3 & 7 & 16 & 54 & 243 & 2038 & 33120 & $\cdots$ \\
\hline
\end{tabular}
\end{center}

A list of representatives for the above equivalence classes for $n\leq 7$ can be found in \cite{vanlint-seidel}.
The value $t_3=2$ corresponds to the Zariski pair found by E. Artal-Bartolo and J. Vall\`es. The value $t_4=3$ corresponds to the Zariski triple found by the authors and H. Tokunaga. 
However, for $n\geq 5$, there are some two-graphs that do not appear in our situation, hence we need to check which equivalence classes actually appear. We were able to check the cases for $n=5,6$ by hand, which gives our second main result.

\begin{thm}\label{thm:main2} Under the notation given above, the numbers of equivalence classes of two-graphs obtained from curves of the form $Q+L_I$ when $|I|=5,6$ are $5, 9$, respectively. Thus we have the following.
\begin{enumerate}
\item There exists a Zariski 5-ple of curves of the form $Q+L_I$ with $|I|=5$.
\item There exists a Zariski 9-ple of curves of the form $Q+L_I$ with $|I|=6$.
\end{enumerate}
\end{thm}

For cases where $n\geq 7$, the discrepancy between the number of equivalence classes of all two-graphs and the number of two-graphs that actually appear in our situation become large, hence a deeper analysis is needed, which we hope to accomplish in the near future.

The organization of this paper is as follows. In Section \ref{sec:two-graphs} we state the definition and basic properties of two-graphs and switching classes. In Section \ref{sec:oursetting} we explain how to associate a two-graph to $Q+L_I$. In Section \ref{sec:proof} we prove Theorem \ref{thm:main1} and Theorem \ref{thm:main2}. 

The authors would like to express their gratitude  to  Professor E. Bannai for informing them about the concept of two-graphs and switching classes, which are the key ingredients in writing this paper. The first author is partially supported by Grant-in-Aid for Scientific Research C (18K03263).

\section{Two-graphs and switching classes}\label{sec:two-graphs}

In this section, we introduce two-graphs and switching classes which will be used in the proof Theorem \ref{thm:main1}. As for details, we refer to \cite{seidel1973}. Let $\Omega=\{\omega_1, \ldots, \omega_n\}$ be a finite set, and let $\Omega^{(3)}$ be the set of all $3$-subsets of $\Omega$. 
\begin{defin}\label{def:two-graph}
A two-graph $(\Omega , \Delta )$ is a pair of a vertex set $\Omega$ and a triple set $\Delta \subset \Omega^{(3)}$, such that each $4$-subset of $\Omega$ contains an even number of triples in $\Delta$.
\end{defin}

Sometimes, we omit $\Delta$ and denote $(\Omega, \Delta)$ simply by $\Omega$ when  $\Delta$ is clear from the context.

\begin{defin}\label{def:equiv-two-graph}
Two two-graphs $(\Omega_1 , \Delta_1)$ and $(\Omega_2 , \Delta_2)$ are said to be equivalent if there exists a bijection $\psi: \Omega_1\rightarrow \Omega_2$ such that $\{\omega_i, \omega_j, \omega_k\}\in \Delta_1$ if and only if $\{\psi(\omega_i), \psi(\omega_j), \psi(\omega_k)\}\in \Delta_2$. 
\end{defin}

\begin{defin}\label{def:sub-two-graph}
An induced sub-two-graph $( \Omega^{\prime} , \Delta^{\prime} )$ of $(\Omega, \Delta)$  is a two-graph  which satisfies $\Omega^{\prime} \subset \Omega$ and $\Delta^{\prime} = \Omega^{ (3)} \cap \Delta$. 
\end{defin}

In order to handle two-graphs easily, we consider switching equivalence on simple graphs. Let $(V , E)$ be a finite simple graph defined by its labeled vertex set $V=\{v_1, \ldots, v_n\}$ and its edge set $E$, where $|V|=n$ for fixed $n$. As is well known, the graph $(V, E)$ can be described by its adjacency matrix, which has the elements zero on the diagonal, $-1$ and $+1$ elsewhere according as the corresponding vertices are adjacent and non-adjacent, respectively. Then, as in \cite{vanlint-seidel}, switching-equivalence is defined as follows: 
\begin{defin}
The graphs $(V, E)$ and $(V^{\prime}, E^{\prime})$ are switching-equivalent if the adjacency matrix $A$ of $(V, E)$ changes to the adjacency matrix $A^{\prime}$ of $(V^{\prime}, E^{\prime})$ by the following two operations:
\begin{itemize}
\item Multiply by $-1$ any row and the corresponding column.
\item Interchange of two rows and, simultaneously, of the corresponding columns.
\end{itemize}
\end{defin}

\begin{thm}[\cite{seidel1973}, Theorem $3. 2$]
Switching-equivalence is an equivalence relation.
\end{thm}

The equivalence classes with respect to switching equivalence will be called {\it switching classes}.

\begin{rem}
The name \lq\lq switching" comes from the fact that the former one of the two operations corresponds to \lq\lq switching"  the existence/non-existence of the edge of the corresponding graph.
\end{rem}

When we consider  switching-equivalence directly  in terms of the graph, the following Lemma holds:

\begin{lem}[\cite{seidel1973}, Lemma 3.9]
Let $(V, E)$ and $(V^{\prime}, E^{\prime})$ be finite simple graphs. Then, the graphs $(V , E)$ and $(V^{\prime}, E^{\prime})$ are switching-equivalent iff there exists a bijection $\psi: V \rightarrow V^{\prime}$ such that  the parity of the number of edges among $\{v_i, v_j, v_k\}$ and $\{\psi(v_i), \psi(v_j), \psi(v_k)\}$ are the same for all  $\{v_i, v_j, v_k\}\in V^{(3)}$.
\end{lem}

For any graph $(V, E)$ with four vertices, it can be easily checked that the number of triples $\{v_i, v_j, v_k\}\subset V$ that have an odd number of edges among them is even. Therefore we can associate a two-graph  $(\Omega_V, \Delta_V)$ to a switching class $(V, E)$ by setting $\Omega_V=V$ and \[\Delta_V=\left\{\{v_i, v_j, v_k\}\in V^{(3)} \mid \text{ $\{v_i, v_j, v_k\}$ has an odd number of edges } \right\} . \]

In fact, this correspondence gives a bijection between switching classes and two-graphs as in the following theorem.

\begin{thm}[\cite{seidel1973}, Theorem $4. 2$]
Given $n$, there is a one-to-one correspondence between  the two-graph structures and the switching classes of graphs on the set of $n$ elements.
\end{thm}



\section{Two-graphs and switching classes associated to configurations of bitangents}\label{sec:oursetting}
In this section, we explain how we associate two-graphs and switching classes with plane curves of the form $Q + L_{I}$ under our setting. 

\subsection{Two-graphs associated to $Q+L_I$}

Let $Q$ be a smooth plane quartic. It is well known that a smooth plane quartic has $28$ bitangent lines $L_1, \ldots, L_{28}$. For $I \subset \{ 1, \ldots , 28 \}$, Put $\mcL_{I}:=\{ L_{i} \mid i \in I \}$.
Let  $\Delta_I$ be the set of triples $\{ L_{i}, L_{j}, L_{k} \}\subset \mcL_I$ such that there exists a conic passing through all six tangent points of $Q  \cap (L_{i}+ L_{j}+ L_{k})$. As will be shown below $(\mcL_I, \Delta_I)$ becomes a two-graph. 
\begin{defin}
The two-graph $\mcL_I=(\mcL_I, \Delta_I)$ is said to be the two-graph associated to $Q+L_I$.
\end{defin}

The existence of a conic is related to the concept of \lq\lq connected numbers" as in  the following Proposition. See \cite{shirane17} or \cite{ban-yam-tok2018}  for details on connected numbers. 
\begin{prop}\label{prop:CtoC}
For all triples $\{ L_{i}, L_{j}, L_{k} \}\subset \mcL_I$, $\{ L_{i}, L_{j}, L_{k} \} \in \Delta_I$ if and only if the connected number of $L_{i}+L_{j}+L_{k}$ is equal to $2$.
\end{prop}

\begin{proof}
The same statement for smooth cubic contact curves instead of $L_i+L_j+L_k$ is proved in  \cite[Proposition 3.3]{bannai2016}. The same proof works in our case where $L_i+L_j+L_k$ is a reducible thus singular contact curve.   
\end{proof}

\begin{thm}
The pair $( \mcL_I , \Delta_I )$ is a two-graph for any $I \subset \{1,\ldots, 28\}$.
\end{thm}

\begin{proof}
 By \cite[Lemma 3.5]{ban-yam-tok2018}, the number of triples in each $4$-subset of $ \mcL_I$ having connected number equal to 2 must be even. Hence, by Proposition  \ref{prop:CtoC} the statement holds. 
\end{proof}

Note that if $I^\prime\subset I$, then $\mcL_{I^\prime}$ can be identified with  an induced sub-two-graph of $\mcL_I$ via the obvious inclusion map.  


\subsection{Switching classes associated to $Q+L_I$ via the $E_{7}^{\ast}$ lattice} 

To a quartic $Q$ and a smooth point $z\in Q$, a rational elliptic surface $S_{Q, z}$ can be associated as described in \cite{tokunaga14}. The properties of $S_{Q, z}$ and its group of sections give valuable data in studying configurations involving $Q$. 
 When $Q$ is a smooth quartic and $z \in Q$ is a general point, by \cite{shioda90}, the Mordell-Weil lattice $\mathrm{MW}(S_{Q , z})$ is isomorphic to the dual root lattice $E_{7}^{\ast}$. The lattice $E_{7}^{\ast}$ can be realized as a sub-lattice of $\RR^{8}$ in a way so that the $56$ minimal vectors of norm $\frac{3}{2}$ in $E_{7}^{\ast}$ are equal to $\pm \frac{1}{4} [ -1, \cdots , -1, 3, 3 ] \in \RR^{8}$ up to permutation. We denote these vectors by
\[
\pm u_{jk} = \pm \frac{1}{4} [ \ \cdots , \ \overset{j}{\check{3}}, \ \cdots , \overset{k}{\check{3}}, \cdots \ ], 
\]
where $1 \leq j < k \leq 8$. 
We also denote the $56$ minimal vectors $\pm u_{jk}$ by $\pm v_{1}, \ldots , \pm v_{28}$ for convenience.
On the other hand, corresponding to the $56$ minimal vectors in $E_{7}^{\ast}$, there are $56$ $\CC(t)$-rational points.
Note that, by \cite{shioda93}, the lines obtained from the $x$-coordinates of the $56$ $\CC(t)$-rational points are bitangents of $Q$. Hence, each pair $\pm v_i$ corresponds to a bitangent $L_i$. Conversely, any bitangent can be obtained by this way. 
 Now for $I \subset \{ 1, \ldots , 28 \}$ with $|I|=n \ (1 \leq n \leq 28)$, 
 by choosing a representative $v_i$ from each pair $\pm v_i$, we obtain a set of  minimal vectors $\{ v_{i} \mid i \in I \}$. Then, we construct a simple graph $(V_{I}, E_{I})$ from $\{ v_{i} \mid i \in I \}$ as follows:

\begin{itemize}
\item Let  $V_{I}=\{ v_{i} \mid i \in I \}$. 
\item Define the edge set $E_{I}$ as follows:
\begin{itemize}
\item If $v_{i_{1}} \cdot v_{i_{2}}=- \frac{1}{2}$ for $v_{i_{1}}, v_{i_{2}} \in V_I$, add an edge $\overline{v_{i_{1}}v_{i_{2}}}$
\item If $v_{i_{1}} \cdot v_{i_{2}}= \frac{1}{2}$ for $v_{i_{1}}, v_{i_{2}} \in V_I$, add no edge between $v_{i_{1}}$ and $v_{i_{2}}$.
\end{itemize}
\end{itemize}

The construction above depends on the choice of  representatives $\{v_i\}$. However, choosing $-v_{i}$ instead of $v_{i}$  corresponds to switching the existence/non-existence of the edges of the corresponding graph as in \S 2. Hence, we have a switching class represented by $(V_I, E_I)$ associated to the set of bitangents $\mcL_I$.

\begin{example}[$n=4$]
Denote each $56$ minimal vectors $\pm u_{jk}$ by $\pm v_{i}$, which satisfies $v_{1}=u_{14}, v_{2}=u_{18}, v_{3}=u_{28}, v_{4}=u_{38}$. Take $I:=\{ 1, 2, 3, 4 \} \subset \{ 1, \ldots , 28 \}$, then we obtain a simple graph $(V_{I}, E_{I})$:
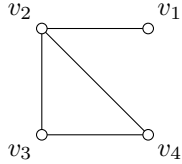
\begin{center}
\begin{tikzpicture}[baseline=1]
\coordinate (A) at (45:1);
\coordinate (B) at (135:1);
\coordinate (C) at (225:1);
\coordinate (D) at (315:1);
\draw (A)--(B);
\draw (B)--(C)--(D)--(B) ;
\draw (A) [fill=white] circle (2pt) node [above right] {$v_{1}$}; 
\draw (B) [fill=white] circle (2pt) node [above left] {$v_{2}$}; 
\draw (C) [fill=white] circle (2pt) node [below left] {$v_{3}$}; 
\draw (D) [fill=white] circle (2pt) node [below right] {$v_{4}$};
\end{tikzpicture} 
\end{center}
The switching class represented by $(V_{I}, E_{I})$ corresponds to the two graph $( \Omega_{V_{I}}, \Delta_{V_{I}} )$ with
\[
\Omega_{V_{I}}= V_{I} , \ \Delta_{V_{I}}= \{ \{ v_{1}, v_{3}, v_{4} \}, \{ v_{2}, v_{3}, v_{4} \} \} . 
\] 
\end{example}

The parity of the number of edges in each triple of the graph $(V_{I}, E_{I})$ is characterised by the following Proposition. 

\begin{prop}\label{prop:StoC}
For $\{ L_{i_{1}}, L_{i_{2}}, L_{i_{3}} \} \subset \mcL$, the number of edges among the corresponding triple $\{ v_{i_{1}}, v_{i_{2}}, v_{i_{3}} \}$ is odd if and only if the connected number of $L_{i_{1}}+L_{i_{2}}+L_{i_{3}}$ equals $2$.   
\end{prop}

\begin{proof}
By \cite{ban-yam-tok2018}, the connected number of $L_{i_{1}}+L_{i_{2}}+L_{i_{3}}$ equals $2$ if and only if the number of edges of the corresponding triple $\{ v_{i_{1}}, v_{i_{2}}, v_{i_{3}} \}$ of the corresponding graph is $1$ or $3$.  
\end{proof}

So far, we have defined two objects associated to $ Q+L_I$, the two-graph $\mcL_I$ and the switching class $(V_I, E_I)$. Actually the two objects are compatible in the following sense.

\begin{prop}
The two graph  corresponding to the switching class represented $(V_I, E_I)$  and the two-graph $\mcL_I$ are equivalent.
\end{prop}
\begin{proof}
This is a direct consequence of Proposition \ref{prop:CtoC} and \ref{prop:StoC}.
\end{proof}

%
%
%

\section{Proof of Main Theorems}\label{sec:proof} 

In this section we prove Theorem \ref{thm:main1} and Theorem \ref{thm:main2}. First, we consider Theorem \ref{thm:main1}.


\begin{prop}\label{prop:homeo2equiv}
Let $I_1, I_2\subset \{1, \ldots, 28\}$ and let $\mcL_{I_1}$, $\mcL_{I_2}$ be their associated two graphs. If there exists a homeomorphism of pairs $h: (\PP^2, Q+L_{I_1})\rightarrow (\PP^2,Q+L_{I_2})$, $h$ induces an equivalence of two-graphs $h_\ast: \mcL_{I_1}\rightarrow \mcL_{I_2}$.
\end{prop}

\begin{proof}
Since $ Q$ is a genus 3 curve and the other irreducible components are lines $h(Q)=Q$ necessarily, hence $h$ induces a bijection $h_\ast:\{L_i\}_{i\in I_1} \rightarrow \{L_j\}_{j\in I_2}$. If $\{L_{i_1}, L_{i_2}, L_{i_3}\}\in \Delta_{I_1}$ then $c_\phi(L_{i_1}+L_{i_2}+L_{i_3})=2$ and by \cite[Proposition 2.1]{ban-yam-tok2018}, $c_\phi(h_\ast(L_{i_1}+L_{i_2}+L_{i_3}))=c_\phi(h_\ast(L_{i_1})+h_\ast(L_{i_2})+h_\ast(L_{i_3}))=2$. Hence $\{h_\ast(L_{i_1}),h_\ast(L_{i_2}), h_\ast(L_{i_3})\}\in \Delta_{I_2}$. This implies that $h_\ast$ is an equivalence of two-graphs.
\end{proof}

The contrapositive of Proposition \ref{prop:homeo2equiv} gives Theorem \ref{thm:main1}.


\bigskip

Next, we consider Theorem  \ref{thm:main2}.
First, we consider the combinatorics of the arrangements of quartics and bitangent lines. For a general quartic, we have the following lemma:
\begin{lem}\label{lem:gen-quart}
Let $Q$ be a general smooth quartic curve. Then the following hold:
\begin{enumerate}
\item Every bitangent line is a proper bitangent, i.e. it is tangent to $Q$ at two distinct points.
\item Any  set of three bitangent lines do not intersect at a single point.
\end{enumerate}
\end{lem}

\begin{proof}
Do to the openness of the conditions, it is enough to find an example satisfying the statements. An example can be constructed using Riemann's Equations for bitangent lines given in \cite[Section 6.1.3]{dolgachev}.
\end{proof}

As a direct consequence of Lemma \ref{lem:gen-quart}, we have the following proposition.

\begin{prop}\label{prop:combinat}
Let $Q$ be a general smooth quartic. Then for a fixed integer $r$ $(1\leq r \leq 28)$, and any $I\subset\{1,\ldots, 28\}$ with $|I|=r$,  the curves of the form
\[
Q+L_I
\]
all have the same combinatorics.
\end{prop}

Since Proposition \ref{prop:combinat} assures that all the curves of the form $Q+L_I$  that we are considering have the same combinatorics, Theorem \ref{thm:main1} implies that the remaining thing that we need in order to prove Theorem \ref{thm:main2}  is the existence of curves with different two-graph structures on $I$. For the existence, we provide concrete examples, but not all possible two-graph structures appear. In order to prove the non-appearance of some of the two-graphs in our case, we prepare some more lemmas. 

\begin{lem}\label{lem:uniquetetrad}
Let $I=\{1, \ldots, 28\}$. For each triple $\{L_{i}, L_{j}, L_{k}\}\in \Delta_I$ there is a unique bitangent $L_{l}$ such that every triple of $L_{\{i, j, k, l\}}$ is an element of $\Delta_I$.
\end{lem}

\begin{proof}
As in the proof of \cite[Proposition 3.3]{bannai2016}, if  there exists a conic $C$ through the six points of tangency of $Q$ and $L_i, L_j, L_k$, the line $L_l$ through the remaining two intersection points of $C$ and $Q$ must be a bitangent line. The existence of the conic implies that the other three triples $\{L_j, L_k, L_l\}$, $\{L_i, L_k, L_l\}$, $\{L_i, L_j, L_l\}$ are also elements of $\Delta_I$. The uniqueness of $L_l$ follows from the uniqueness of a conic passing through six points. 
\end{proof}

\begin{remark}
Note that a bitangent line of $Q$ corresponds to an odd theta characteristic of $Q$. Lemma \ref{lem:uniquetetrad} is equivalent to the fact the a syzygetic triad of theta characteristics is contained in a unique syzygetic tetrad.
\end{remark}

\begin{figure}
\centering
\hfill
\begin{tikzpicture}[baseline=1]
\coordinate (A) at (90:1);
\coordinate (B) at (162:1);
\coordinate (C) at (235:1);
\coordinate (D) at (307:1);
\coordinate (E) at (18:1);
\draw (A) --(B)--(C)--(A);
\draw (A) [fill=white] circle (2pt); 
\draw (B) [fill=white] circle (2pt); 
\draw (C) [fill=white] circle (2pt); 
\draw (D) [fill=white] circle (2pt);
\draw (E) [fill=white] circle (2pt);  
\node at (A) [above] {$v_1$};
\node at (B) [left] {$v_2$};
\node at (C) [below] {$v_3$};
\node at (D) [below] {$v_4$};
\node at (E) [right] {$v_5$};
\end{tikzpicture}
\hfill
\begin{tikzpicture}[baseline=1]
\coordinate (A) at (90:1);
\coordinate (B) at (162:1);
\coordinate (C) at (235:1);
\coordinate (D) at (307:1);
\coordinate (E) at (18:1);
\draw (A) --(B)--(C)--(D)--(E)--(A)--(C)--(E)--(B)--(D)--(A);
\draw (A) [fill=white] circle (2pt); 
\draw (B) [fill=white] circle (2pt); 
\draw (C) [fill=white] circle (2pt); 
\draw (D) [fill=white] circle (2pt);
\draw (E) [fill=white] circle (2pt);  
\node at (A) [above] {$v_1$};
\node at (B) [left] {$v_2$};
\node at (C) [below] {$v_3$};
\node at (D) [below] {$v_4$};
\node at (E) [right] {$v_5$};
\end{tikzpicture} 
\hfill \hspace{0.1pt}
\caption{Representatives of switching classes corresponding to two-graphs with $|\Omega|=5$ and $|\Delta|=7$ (left), $|\Delta|=10$ (right).}
\label{n=5d>6}
\end{figure}
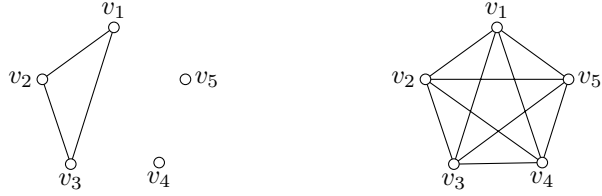

\begin{lem}\label{lem:d_leq_6}
 For any $I\subset \{1,\ldots, 28\}$ with $|I|=5$, its associated two graph $\mcL_I=(\mcL_{I} , \Delta_{I})$ must satisify $|\Delta_{I}|\leq 6$.
\end{lem}

\begin{proof}
From the classification of two-graphs on 5 elements, the possible values for $|\Delta_{I}|$ are $|\Delta_{I}|=0,3,4,5,6,7,10$. The cases $|\Delta_{I}|=7,10$ can be represented by the switching class in Figure \ref{n=5d>6}.
In  both cases, $\{v_1, v_2, v_3, v_4\}$ and $\{v_1, v_2, v_3, v_5\}$ are tetrads containing $\{v_1, v_2, v_3\}$ such that every triple in them is contained in $\Delta_{I}$. This contradicts the uniqueness of Lemma \ref{lem:uniquetetrad}. Hence $|\Delta_{I}|\not=7, 10$ and we have $|\Delta_{I}|\leq 6$.
\end{proof}

\begin{cor}\label{cor:d_leq_6}
For any $I\subset\{1, \ldots, 28\}$, its associated two graph $\mcL_{I_1}$ cannot have an induced sub-two-graph $(\Omega, \Delta)$ such that $|\Omega|=5$, $|\Delta|> 6$.
\end{cor}

\begin{lem}\label{lem:reduction}
Let $n\geq 4$, $(\Omega, \Delta)$ be a two graph with $\Omega=\{v_1,\ldots, v_n\}$ and let $(\Omega_i, \Delta_i)$ $(i=1,\ldots,n)$ be the induced 
sub-two-graphs of $(\Omega, \Delta)$ obtained by deleting $v_i$ $(i=1,\ldots, n \,\text{resp.})$. Let $|\Delta|=d$ and $|\Delta_i|=d_i$. Then there exists at least one $i$ such that $d_i\geq\frac{n-3}{n}d$.
\end{lem}

\begin{proof}
Since each triple $\omega\in \Delta$ is contained in $n-3$ of the sub-two-graphs, we have
\[
(n-3)d=\sum_{i=1}^nd_i.
\]
If $d_i<\dfrac{n-3}{n}d$ for all $i$,  then $\displaystyle\sum_{i=1}^nd_i<(n-3)d$ which is a contradiction. Hence, there must be a $d_i$ satisfying $d_i\geq\frac{n-3}{n} d$.
\end{proof}

\begin{cor}\label{cor:d<14}
For any subset $I_1\subset\{1, \ldots, 28\}$ with $|I_1|=6$, $|\Delta_{I_1}|<14$.
\end{cor}
\begin{proof}
If $|\Delta_{I_1}|=d\geq14$, then by Lemma \ref{lem:reduction}, there must exist a indused sub-two-graph with $|\Delta_i|=d_i\geq 7$ which is impssible by Lemma \ref{cor:d_leq_6}.
\end{proof}

Now we are ready to prove Theorem \ref{thm:main2}.
First, we consider the case where $|I|=5$.
Among the switching classes in Figure \ref{n=5},  we have proved in Lemma \ref{lem:d_leq_6} that the cases $(5,7)$, $(5,10)$ cannot appear. For the other cases, we can find  explicit combinations of minimal vectors $\pm u_{ij}\in E_7^\ast$ giving the desired switching classes, for example:
\begin{align*}
&(5,0):  u_{18}, u_{28}, u_{38}, u_{48}, u_{58} &(5,3):  u_{18}, u_{28}, u_{38}, u_{48}, -u_{15}\\
&(5,4):  u_{18}, u_{28}, u_{38}, u_{48}, -u_{12}&(5,5):  u_{18}, u_{28}, u_{38}, u_{23}, -u_{24}\\
&(5,6):  u_{18}, u_{28}, u_{13}, u_{23}, u_{12}
\end{align*}

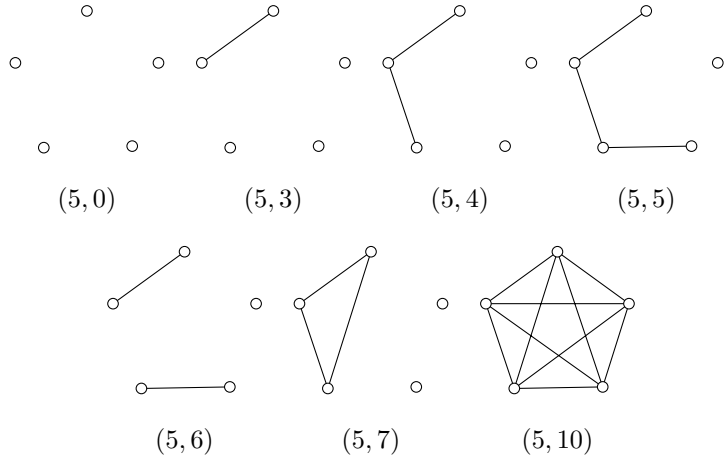
\begin{figure}[h]
\centering
\begin{tabular}{cccc}
\begin{tikzpicture}[baseline=1]
\coordinate (A) at (90:1);
\coordinate (B) at (162:1);
\coordinate (C) at (235:1);
\coordinate (D) at (307:1);
\coordinate (E) at (18:1);
\draw (A) [fill=white] circle (2pt); 
\draw (B) [fill=white] circle (2pt); 
\draw (C) [fill=white] circle (2pt); 
\draw (D) [fill=white] circle (2pt);
\draw (E) [fill=white] circle (2pt);  
\end{tikzpicture}
&
\begin{tikzpicture}[baseline=1]
\coordinate (A) at (90:1);
\coordinate (B) at (162:1);
\coordinate (C) at (235:1);
\coordinate (D) at (307:1);
\coordinate (E) at (18:1);
\draw (A) --(B);
\draw (A) [fill=white] circle (2pt); 
\draw (B) [fill=white] circle (2pt); 
\draw (C) [fill=white] circle (2pt); 
\draw (D) [fill=white] circle (2pt);
\draw (E) [fill=white] circle (2pt);  
\end{tikzpicture}
&
\begin{tikzpicture}[baseline=1]
\coordinate (A) at (90:1);
\coordinate (B) at (162:1);
\coordinate (C) at (235:1);
\coordinate (D) at (307:1);
\coordinate (E) at (18:1);
\draw (A) --(B)--(C);
\draw (A) [fill=white] circle (2pt); 
\draw (B) [fill=white] circle (2pt); 
\draw (C) [fill=white] circle (2pt); 
\draw (D) [fill=white] circle (2pt);
\draw (E) [fill=white] circle (2pt);  
\end{tikzpicture}
&
\begin{tikzpicture}[baseline=1]
\coordinate (A) at (90:1);
\coordinate (B) at (162:1);
\coordinate (C) at (235:1);
\coordinate (D) at (307:1);
\coordinate (E) at (18:1);
\draw (A) --(B)--(C)--(D);
\draw (A) [fill=white] circle (2pt); 
\draw (B) [fill=white] circle (2pt); 
\draw (C) [fill=white] circle (2pt); 
\draw (D) [fill=white] circle (2pt);
\draw (E) [fill=white] circle (2pt);  
\end{tikzpicture} 
\\
&&&\\
$(5,0)$ & $(5,3)$ & $(5,4)$ & $(5,5)$ 
\end{tabular}
\vspace{10pt}

\begin{tabular}{ccc}
\begin{tikzpicture}[baseline=1]
\coordinate (A) at (90:1);
\coordinate (B) at (162:1);
\coordinate (C) at (235:1);
\coordinate (D) at (307:1);
\coordinate (E) at (18:1);
\draw (A) --(B);
\draw (C)--(D);
\draw (A) [fill=white] circle (2pt); 
\draw (B) [fill=white] circle (2pt); 
\draw (C) [fill=white] circle (2pt); 
\draw (D) [fill=white] circle (2pt);
\draw (E) [fill=white] circle (2pt);  
\end{tikzpicture}
&
\begin{tikzpicture}[baseline=1]
\coordinate (A) at (90:1);
\coordinate (B) at (162:1);
\coordinate (C) at (235:1);
\coordinate (D) at (307:1);
\coordinate (E) at (18:1);
\draw (A) --(B)--(C)--(A);
\draw (A) [fill=white] circle (2pt); 
\draw (B) [fill=white] circle (2pt); 
\draw (C) [fill=white] circle (2pt); 
\draw (D) [fill=white] circle (2pt);
\draw (E) [fill=white] circle (2pt);  
\end{tikzpicture}
&
\begin{tikzpicture}[baseline=1]
\coordinate (A) at (90:1);
\coordinate (B) at (162:1);
\coordinate (C) at (235:1);
\coordinate (D) at (307:1);
\coordinate (E) at (18:1);
\draw (A) --(B)--(C)--(D)--(E)--(A)--(C)--(E)--(B)--(D)--(A);
\draw (A) [fill=white] circle (2pt); 
\draw (B) [fill=white] circle (2pt); 
\draw (C) [fill=white] circle (2pt); 
\draw (D) [fill=white] circle (2pt);
\draw (E) [fill=white] circle (2pt);  
\end{tikzpicture}
\\
&&\\
$(5,6)$ & $(5,7)$ & $(5,10)$ 
\end{tabular}
\caption{Representatives of switching classes corresponding to two graphs with $|\Omega|=5$ labeled by $(|\Omega|, |\Delta|)$.} 
\label{n=5}
\end{figure}

Next, for the case $|I|=6$, 
among the switching classes in Figure \ref{n=6}, $(6, 14)$, $(6,16)$, $(6,20)$ cannot appear due to Corollary \ref{cor:d<14}. Also we can easily check that the switching classes $(6,10)_2$, $(6, 10)_3$, $(6,12)_1$, $(6,12)_2$ contain an induced sub switching class equivalent to $(5,7)$, hence these cannot appear either due to Corollary \ref{cor:d_leq_6}. (We could have done the same for $(6, 14)$, $(6,16)$, $(6,20)$.) For the remaining cases $(6,0)$, $(6,4)$, $(6,6)$, $(6,8)_1$, $(6,8)_2$, $(6,8)_3$, $(6,10)_1$, $(6,10)_4$ and $(6,12)_3$ we can find explicit combinations of minimal vectors $\pm u_{ij}\in E_7^\ast$ giving the desired switching classes, for example:
\begin{align*}
& (6,0): u_{18}, u_{28}, u_{38}, u_{48}, u_{58}, u_{68}
&(&6,4): u_{18}, u_{28}, u_{38}, u_{48}, u_{58}, -u_{16}\\
& (6,6): u_{18}, u_{28}, u_{38}, u_{48}, u_{58}, -u_{12}
&(&6, 8)_1: u_{18}, u_{28}, u_{38}, u_{48}, -u_{15}, -u_{26}\\
& (6,8)_2: u_{18}, u_{28}, u_{38}, u_{48}, -u_{15}, -u_{25}
&(&6,8)_3: u_{18} , u_{28}, u_{38}, -u_{14}, u_{23}, u_{48}\\
& (6,10)_1: u_{18}, u_{28}, u_{38}, u_{48}, -u_{14}, -u_{34}
&(&6,10)_4: u_{18}, u_{28}, u_{38}, u_{23}, -u_{24}, u_{35}\\
& (6, 12)_3: u_{18}, u_{28}, u_{38}, u_{12}, u_{13}, u_{23} &
\end{align*}

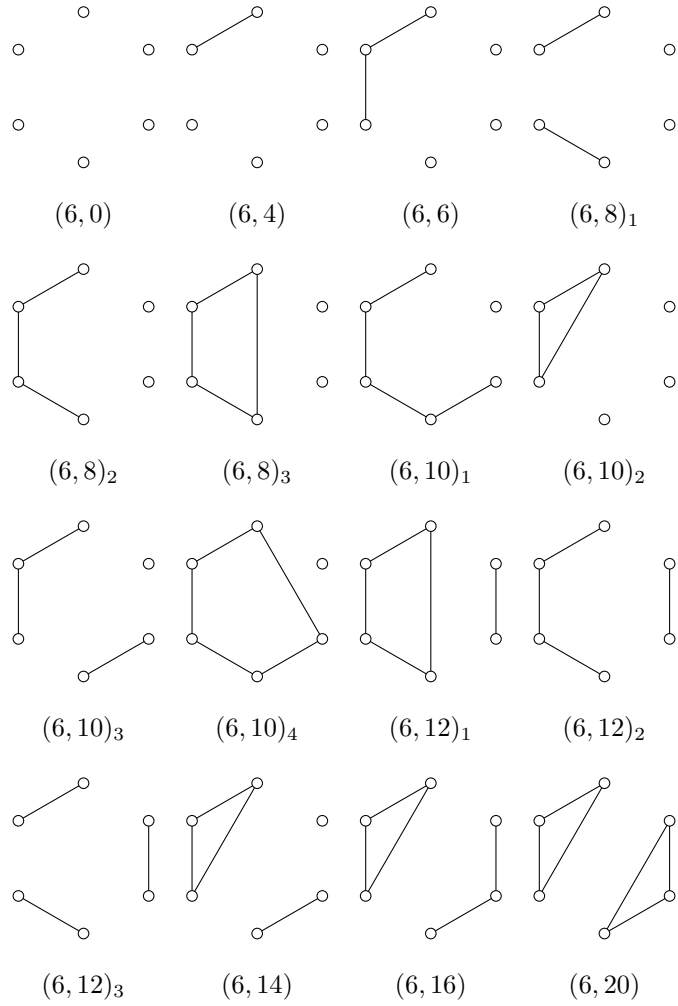
\begin{figure}[h]
\centering
\begin{tabular}{cccc}
\begin{tikzpicture}[baseline=1]
\coordinate (A) at (90:1);
\coordinate (B) at (150:1);
\coordinate (C) at (210:1);
\coordinate (D) at (270:1);
\coordinate (E) at (330:1);
\coordinate (F) at (30:1);
\draw (A) [fill=white] circle (2pt); 
\draw (B) [fill=white] circle (2pt); 
\draw (C) [fill=white] circle (2pt); 
\draw (D) [fill=white] circle (2pt);
\draw (E) [fill=white] circle (2pt);  
\draw (F) [fill=white] circle (2pt);  
\end{tikzpicture}
&
\begin{tikzpicture}[baseline=1]
\coordinate (A) at (90:1);
\coordinate (B) at (150:1);
\coordinate (C) at (210:1);
\coordinate (D) at (270:1);
\coordinate (E) at (330:1);
\coordinate (F) at (30:1);
\draw (A) -- (B);
\draw (A) [fill=white] circle (2pt); 
\draw (B) [fill=white] circle (2pt); 
\draw (C) [fill=white] circle (2pt); 
\draw (D) [fill=white] circle (2pt);
\draw (E) [fill=white] circle (2pt);  
\draw (F) [fill=white] circle (2pt);  
\end{tikzpicture}
&
\begin{tikzpicture}[baseline=1]
\coordinate (A) at (90:1);
\coordinate (B) at (150:1);
\coordinate (C) at (210:1);
\coordinate (D) at (270:1);
\coordinate (E) at (330:1);
\coordinate (F) at (30:1);
\draw (A) -- (B) --(C);
\draw (A) [fill=white] circle (2pt); 
\draw (B) [fill=white] circle (2pt); 
\draw (C) [fill=white] circle (2pt); 
\draw (D) [fill=white] circle (2pt);
\draw (E) [fill=white] circle (2pt);  
\draw (F) [fill=white] circle (2pt);  
\end{tikzpicture}
&
\begin{tikzpicture}[baseline=1]
\coordinate (A) at (90:1);
\coordinate (B) at (150:1);
\coordinate (C) at (210:1);
\coordinate (D) at (270:1);
\coordinate (E) at (330:1);
\coordinate (F) at (30:1);
\draw (A) -- (B);
\draw (C) --(D);
\draw (A) [fill=white] circle (2pt); 
\draw (B) [fill=white] circle (2pt); 
\draw (C) [fill=white] circle (2pt); 
\draw (D) [fill=white] circle (2pt);
\draw (E) [fill=white] circle (2pt);  
\draw (F) [fill=white] circle (2pt);  
\end{tikzpicture}
\\
&&&\\
$(6,0)$ & $(6,4)$ & $(6,6)$ & $(6,8)_1$ \\
&&&\\
\begin{tikzpicture}[baseline=1]
\coordinate (A) at (90:1);
\coordinate (B) at (150:1);
\coordinate (C) at (210:1);
\coordinate (D) at (270:1);
\coordinate (E) at (330:1);
\coordinate (F) at (30:1);
\draw (A) -- (B)--(C)--(D);
\draw (A) [fill=white] circle (2pt); 
\draw (B) [fill=white] circle (2pt); 
\draw (C) [fill=white] circle (2pt); 
\draw (D) [fill=white] circle (2pt);
\draw (E) [fill=white] circle (2pt);  
\draw (F) [fill=white] circle (2pt);  
\end{tikzpicture}
&
\begin{tikzpicture}[baseline=1]
\coordinate (A) at (90:1);
\coordinate (B) at (150:1);
\coordinate (C) at (210:1);
\coordinate (D) at (270:1);
\coordinate (E) at (330:1);
\coordinate (F) at (30:1);
\draw (A) -- (B)--(C)--(D)--(A);
\draw (A) [fill=white] circle (2pt); 
\draw (B) [fill=white] circle (2pt); 
\draw (C) [fill=white] circle (2pt); 
\draw (D) [fill=white] circle (2pt);
\draw (E) [fill=white] circle (2pt);  
\draw (F) [fill=white] circle (2pt);  
\end{tikzpicture}
&
\begin{tikzpicture}[baseline=1]
\coordinate (A) at (90:1);
\coordinate (B) at (150:1);
\coordinate (C) at (210:1);
\coordinate (D) at (270:1);
\coordinate (E) at (330:1);
\coordinate (F) at (30:1);
\draw (A) -- (B)--(C)--(D)--(E);
\draw (A) [fill=white] circle (2pt); 
\draw (B) [fill=white] circle (2pt); 
\draw (C) [fill=white] circle (2pt); 
\draw (D) [fill=white] circle (2pt);
\draw (E) [fill=white] circle (2pt);  
\draw (F) [fill=white] circle (2pt);  
\end{tikzpicture}
&
\begin{tikzpicture}[baseline=1]
\coordinate (A) at (90:1);
\coordinate (B) at (150:1);
\coordinate (C) at (210:1);
\coordinate (D) at (270:1);
\coordinate (E) at (330:1);
\coordinate (F) at (30:1);
\draw (A) -- (B)--(C)--(A);
\draw (A) [fill=white] circle (2pt); 
\draw (B) [fill=white] circle (2pt); 
\draw (C) [fill=white] circle (2pt); 
\draw (D) [fill=white] circle (2pt);
\draw (E) [fill=white] circle (2pt);  
\draw (F) [fill=white] circle (2pt);  
\end{tikzpicture}
\\
&&&\\
 $(6,8)_2$ & $(6,8)_3$ & $(6,10)_1$ & $(6,10)_2$ \\
&&&\\ 
\begin{tikzpicture}[baseline=1]
\coordinate (A) at (90:1);
\coordinate (B) at (150:1);
\coordinate (C) at (210:1);
\coordinate (D) at (270:1);
\coordinate (E) at (330:1);
\coordinate (F) at (30:1);
\draw (A) --(B)--(C);
\draw (D)--(E);
\draw (A) [fill=white] circle (2pt); 
\draw (B) [fill=white] circle (2pt); 
\draw (C) [fill=white] circle (2pt); 
\draw (D) [fill=white] circle (2pt);
\draw (E) [fill=white] circle (2pt);  
\draw (F) [fill=white] circle (2pt);  
\end{tikzpicture}
&
\begin{tikzpicture}[baseline=1]
\coordinate (A) at (90:1);
\coordinate (B) at (150:1);
\coordinate (C) at (210:1);
\coordinate (D) at (270:1);
\coordinate (E) at (330:1);
\coordinate (F) at (30:1);
\draw (A) --(B)--(C)--(D)--(E)--(A);
\draw (A) [fill=white] circle (2pt); 
\draw (B) [fill=white] circle (2pt); 
\draw (C) [fill=white] circle (2pt); 
\draw (D) [fill=white] circle (2pt);
\draw (E) [fill=white] circle (2pt);  
\draw (F) [fill=white] circle (2pt);  
\end{tikzpicture}
&
\begin{tikzpicture}[baseline=1]
\coordinate (A) at (90:1);
\coordinate (B) at (150:1);
\coordinate (C) at (210:1);
\coordinate (D) at (270:1);
\coordinate (E) at (330:1);
\coordinate (F) at (30:1);
\draw (A) -- (B)--(C)--(D)--(A);
\draw (E)--(F);
\draw (A) [fill=white] circle (2pt); 
\draw (B) [fill=white] circle (2pt); 
\draw (C) [fill=white] circle (2pt); 
\draw (D) [fill=white] circle (2pt);
\draw (E) [fill=white] circle (2pt);  
\draw (F) [fill=white] circle (2pt);  
\end{tikzpicture}
&
\begin{tikzpicture}[baseline=1]
\coordinate (A) at (90:1);
\coordinate (B) at (150:1);
\coordinate (C) at (210:1);
\coordinate (D) at (270:1);
\coordinate (E) at (330:1);
\coordinate (F) at (30:1);
\draw (A) -- (B)--(C)--(D);
\draw (E)--(F);
\draw (A) [fill=white] circle (2pt); 
\draw (B) [fill=white] circle (2pt); 
\draw (C) [fill=white] circle (2pt); 
\draw (D) [fill=white] circle (2pt);
\draw (E) [fill=white] circle (2pt);  
\draw (F) [fill=white] circle (2pt);  
\end{tikzpicture}
\\
&&&\\
 $(6,10)_3$ & $(6,10)_4$ & $(6,12)_1$ & $(6,12)_2$\\
 &&&\\
\begin{tikzpicture}[baseline=1]
\coordinate (A) at (90:1);
\coordinate (B) at (150:1);
\coordinate (C) at (210:1);
\coordinate (D) at (270:1);
\coordinate (E) at (330:1);
\coordinate (F) at (30:1);
\draw (A) -- (B);
\draw (C)--(D);
\draw (E)--(F);
\draw (A) [fill=white] circle (2pt); 
\draw (B) [fill=white] circle (2pt); 
\draw (C) [fill=white] circle (2pt); 
\draw (D) [fill=white] circle (2pt);
\draw (E) [fill=white] circle (2pt);  
\draw (F) [fill=white] circle (2pt);  
\end{tikzpicture}
&
\begin{tikzpicture}[baseline=1]
\coordinate (A) at (90:1);
\coordinate (B) at (150:1);
\coordinate (C) at (210:1);
\coordinate (D) at (270:1);
\coordinate (E) at (330:1);
\coordinate (F) at (30:1);
\draw (A) -- (B)--(C)--(A);
\draw (D)--(E);
\draw (A) [fill=white] circle (2pt); 
\draw (B) [fill=white] circle (2pt); 
\draw (C) [fill=white] circle (2pt); 
\draw (D) [fill=white] circle (2pt);
\draw (E) [fill=white] circle (2pt);  
\draw (F) [fill=white] circle (2pt);  
\end{tikzpicture}
&
\begin{tikzpicture}[baseline=1]
\coordinate (A) at (90:1);
\coordinate (B) at (150:1);
\coordinate (C) at (210:1);
\coordinate (D) at (270:1);
\coordinate (E) at (330:1);
\coordinate (F) at (30:1);
\draw (A) -- (B)--(C)--(A);
\draw (D)--(E)--(F);
\draw (A) [fill=white] circle (2pt); 
\draw (B) [fill=white] circle (2pt); 
\draw (C) [fill=white] circle (2pt); 
\draw (D) [fill=white] circle (2pt);
\draw (E) [fill=white] circle (2pt);  
\draw (F) [fill=white] circle (2pt);  
\end{tikzpicture}
&
\begin{tikzpicture}[baseline=1]
\coordinate (A) at (90:1);
\coordinate (B) at (150:1);
\coordinate (C) at (210:1);
\coordinate (D) at (270:1);
\coordinate (E) at (330:1);
\coordinate (F) at (30:1);
\draw (A) -- (B)--(C)--(A);
\draw (D)--(E)--(F)--(D);
\draw (A) [fill=white] circle (2pt); 
\draw (B) [fill=white] circle (2pt); 
\draw (C) [fill=white] circle (2pt); 
\draw (D) [fill=white] circle (2pt);
\draw (E) [fill=white] circle (2pt);  
\draw (F) [fill=white] circle (2pt);  
\end{tikzpicture}\\
&&&\\
$(6,12)_3$ & $(6,14)$ & $(6, 16)$ & $(6,20)$
\end{tabular}
\caption{Representatives of switching classes corresponding to two graphs with $|\Omega|=6$ labeled by $(|\Omega|, |\Delta|)$.} 
\label{n=6}
\end{figure}

\noindent Shinzo BANNAI\\
National Institute of Technology, Ibaraki College, 866 Nakane, Hitachinaka-shi, Ibaraki-Ken 312-8508 JAPAN, 
{\tt sbannai@ge.ibaraki-ct.ac.jp}\\

\noindent Momoko YAMAMOTO-OHNO\\
Department of Mathematics Sciences\\
Tokyo Metropolitan University, 1-1 Minami-Ohsawa, Hachiohji 192-0397 JAPAN,\\ 
{\tt yamamoto-momoko@ed.tmu.ac.jp}
%
%

\end{document}